\documentclass[12pt]{amsart}
\usepackage[latin9]{inputenc}
\usepackage{color}
\usepackage{verbatim}
\usepackage{amsthm}
\usepackage{amssymb}
\usepackage{multicol}
\usepackage{tabularx}
\usepackage{graphicx}
\newcounter{next}

\makeatletter

\def\qed{\hfill
\ifhmode\unskip\nobreak\fi\quad\ifmmode\Box\else$\Box$\fi\\ }
\newtheorem{theorem}{Theorem}
\newtheorem{cor}[theorem]{Corollary}

\numberwithin{equation}{section}
\numberwithin{figure}{section}
\usepackage{enumitem}		
    \newtheorem{question}[theorem]{\protect\questionname}

\usepackage{fullpage}
\newtheorem{thm}[theorem]{Theorem}
\usepackage{tikz}
\usepackage{pgf}
\newcount\mycount

\usepackage{caption}
\usepackage[labelformat=simple,labelfont={}]{subcaption}

\usetikzlibrary{shapes}
\usepackage{graphicx}
\usepackage{currfile}

\makeatother

    \providecommand{\questionname}{Question}

\newcommand{\s}{s}
\newcommand{\ms}{\mathcal S}
\begin{document}



\title{The $(2k-1)$-connected multigraphs with at most $k-1$ disjoint cycles}\thanks{The first two
 authors thank Institut Mittag-Leffler (Djursholm, Sweden) for the hospitality and creative environment.}

\author{H.A. Kierstead}

\thanks{Department of Mathematics and Statistics, Arizona State University,
Tempe, AZ 85287, USA. E-mail address: kierstead@asu.edu. Research
of this author is supported in part by NSA grant H98230-12-1-0212.}

\author{A.V. Kostochka}
\thanks{Department of Mathematics, University of Illinois, Urbana, IL, 61801,
USA and Sobolev Institute of Mathematics, Novosibirsk, Russia. E-mail address:
kostochk@math.uiuc.edu. Research of this author is supported in part
by NSF grant   DMS-1266016 and by Grant NSh.1939.2014.1 of the President of
Russia for Leading Scientific Schools.}

\author{E.C. Yeager}
\thanks{Department of Mathematics, University of Illinois, Urbana, IL, 61801,
USA. E-mail address: \linebreak yeager2@illinois.edu. Research of this author is supported in part
by NSF grants  DMS 08-38434 and DMS-1266016.}
\begin{abstract}

In 1963, Corr\' adi and Hajnal proved that for all $k\geq1$ and $n\geq3k$,
every (simple) graph $G$ on $n$ vertices with minimum degree $\delta(G)\geq2k$
contains $k$ disjoint cycles.
The same year, Dirac described the $3$-connected multigraphs not containing two disjoint cycles and asked the more general question:
 Which $(2k-1)$-connected multigraphs do not contain $k$ disjoint cycles?
Recently, the authors characterized the simple graphs $G$ with minimum degree $\delta(G) \geq 2k-1$ that do not contain $k$ disjoint cycles.
We use this result to answer Dirac's question in full.
\end{abstract}

\maketitle
\noindent{\small{Mathematics Subject Classification: 05C15, 05C35, 05C40.}}{\small \par}

\noindent{\small{Keywords: Disjoint cycles, connected graphs, graph packing, equitable coloring, minimum degree.}}{\small \par}



\section{Introduction}
 For a multigraph $G=(V,E)$, let $|G|=|V|$, $\|G\|=|E|$, $\delta(G)$ be the minimum degree of  $G$, and $\alpha(G)$ 
be the independence number of $G$. In this note, we allow multigraphs to have loops as well as multiple edges.
For a simple graph $G$, let $\overline G$ denote the complement of $G$ and
for disjoint graphs $G$ and $H$, let $G\vee H$ denote $G\cup H$ together with all edges from $V(G)$ to $V(H)$.

In 1963, Corr\' adi and Hajnal proved a conjecture of Erd\H os by showing the following:
\begin{thm}[\cite{CH}]
 \label{cht} Let $k\in\mathbb{Z}^{+}$. Every graph $G$ with
$|G|\geq3k$ and $\delta(G)\geq2k$ contains $k$ disjoint cycles.
\end{thm}

The hypothesis $\delta(G)\geq2k$ is best possible, as shown by the $3k$-vertex graph \linebreak
$H=\overline{K}_{k+1}\vee K_{2k-1}$, which has $\delta(H)=2k-1$ but does not contain $k$ disjoint cycles.
Recently, the authors refined Theorem~\ref{cht} by characterizing all simple 
 graphs
that fulfill the weaker hypothesis $\delta(G) \geq 2k-1$ and contain $k$ disjoint cycles. This refinement depends on an extremal graph $Y_{k,k}$. 

Let 
$Y_{h,t}= \overline{K}_h\vee(K_t\cup K_t)$ (Figure~\ref{Yht}),
where 
  $V(\overline{K}_h)=X_0$ and the cliques have vertex sets $X_1$ and $X_2$.
In other words, $V(Y_{h,t})=X_0\cup X_1\cup X_2$ with $|X_0|=h$ and $|X_1|=|X_2|=t$, and
a pair $xy$ is an edge in $Y_{h,t}$ iff $\{x,y\}\subseteq X_1$, or $\{x,y\}\subseteq X_2$, or
$|\{x,y\}\cap X_0|=1$.

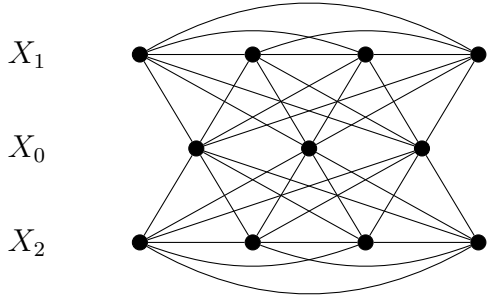
\begin{figure}[ht]
\begin{tikzpicture}
\foreach \x in {1,...,4}{
	\draw (1.5*\x,1.75) node[shape=circle,fill=black,inner sep=0pt,minimum size=2mm,draw] (v\x1) {};
\draw (1.5*\x,4.25) node[shape=circle,fill=black,inner sep=0pt,minimum size=2mm,draw] (v\x3) {};
	}
\foreach \x in {2,3,4}{
	\draw (-.75+1.5*\x,1.5*2) node[shape=circle,fill=black,inner sep=0pt,minimum size=2mm,draw] (v\x2) {};}
\foreach \y in {2,3,4}
	\foreach \x in {1,...,4}{
		\draw (v\x1) -- (v\y2);}
\foreach \y in {1,...,4}
	\foreach \x in {2,3,4}{
		\draw (v\x2) -- (v\y3);}
\draw (v11)--(v21)--(v31)--(v41);
\draw (v13)--(v23)--(v33)--(v43);
\draw (v31) to[out=200, in=-20] (v11);
\draw (v41) to[out=200, in=-20] (v21);
\draw (v41) to[out=210, in=-30] (v11);
\draw (v33) to[out=-200, in=20] (v13);
\draw (v43) to[out=-200, in=20] (v23);
\draw (v43) to[out=-210, in=30] (v13);
\draw (0,1.75) node{$X_2$};
\draw (0,3) node{$X_0$};
\draw (0,4.25) node{$X_1$};
\end{tikzpicture}\caption{$Y_{h,t}$, shown with $h=3$ and $t=4$.}\label{Yht}
\end{figure}

\begin{thm}[\cite{KK2}]
\label{ch++}Let $k\geq 2$. Every simple 
graph $G$ with $|G|\geq3k$
and $\delta(G)\geq2k-1$ contains $k$ disjoint cycles if and
only if:
\begin{enumerate}[label=(\roman*), ref=(\alph*)] 
\item $\alpha(G)\leq|G|-2k$;
\item if $k$ is odd and $|G|=3k$, then $G\neq Y_{k,k} $; and
\item if $k=2$ then $G$ is not a wheel.
\end{enumerate}
\end{thm}

 Extending Theorem~\ref{cht}, Dirac and Erd\H os~\cite{DE} showed that if a graph $G$ has many more vertices of degree at least $2k$ than vertices of
lower degree, then $G$ has $k$ disjoint cycles. 
\begin{thm}[\cite{DE}]
If $G$ is a simple graph and $k \geq 3$, and if the number of vertices in $G$ with degree at least $2k$ exceeds the number of vertices with degree at most $2k-2$ by at least $k^2+2k-4$, then $G$ contains $k$ disjoint cycles.
\end{thm}

Dirac~\cite{Di} described all $3$-connected multigraphs that do not have two disjoint cycles and posed
 the following question:
\begin{question}[\cite{Di}]\label{DiracQ}
 Which $(2k-1)$-connected multigraphs\footnote{Dirac used the word {\em graphs}, but in~\cite{Di} this appears to mean {\em multigraphs}.} do not have $k$ disjoint cycles?
 \end{question}

We consider the class $\mathcal D_k$ of multigraphs in which each vertex has at least $2k-1$ distinct neighbors. Our main result, Theorem~\ref{dm}, characterizes those multigraphs in $\mathcal D_k$ that do not contain $k$ disjoint cycles. Every ($2k-1$)-connected multigraph is in $D_k$, so this provides a complete answer to Question~\ref{DiracQ}. Determining whether a multigraph is in $\mathcal D_k$, and determining whether a multigraph is ($2k-1$)-connected, can be accomplished in polynomial time.

In the next section, we introduce notation,
discuss existing results to be used later on,
 and state our main result, Theorem~\ref{dm}. 
In the last two sections, we prove Theorem~\ref{dm}.

\section{Preliminaries and statement of the main result}
\subsection{Notation}
For every multigraph $G$, let  $V_1=V_1(G)$ be the set of vertices in $G$ incident to loops.
Let $\widetilde{G}$ denote the {\em underlying simple graph of $G$}, i.e. the
simple graph on $V(G)$ such that two vertices are adjacent in $G$ if and only if they are adjacent in $\widetilde G$. Let
 $F=F(G)$ be the simple graph formed by the multiple edges in $G-V_1$;
 that is, if $G'$ is the subgraph of $G-V_1$ induced by its multiple edges, then 
 $F=\widetilde{G'}$.
We will call the edges of $F(G)$ {\em the strong edges of } $G$, and 
define $\alpha'=\alpha'(F)$ to be the size of a maximum matching in $F$.
A set $S=\{v_0,\ldots,v_s\}$ of vertices in a graph $H$ is a {\em superstar with center $v_0$ in $H$} if
$N_H(v_i)=\{v_0\}$ for each $1\leq i\leq s$ and $H-S$ has a perfect matching.

For $v \in V$, we define $\s(v)=|N(v)|$ to be the \emph{simple degree} of $v$, and we say that
$\ms(G)=\min\{\s(v):v \in V\}$ is the \emph{minimum simple degree} of $G$.
We define $\mathcal{D}_k$ to be the family of multigraphs $G$ with $\ms(G) \geq 2k-1$.
By the definition of $\mathcal{D}_k$, $\alpha(G)\leq n-2k+1$ for every $n$-vertex $G\in\mathcal{D}_k$;
so we call $G\in\mathcal{D}_k$ {\em extremal} if $\alpha(G)= n-2k+1$.
A {\em big set} in an extremal $G\in\mathcal{D}_k$
 is an independent set of size $\alpha(G)$. If $I$ is a big set in an extremal $G\in\mathcal{D}_k$, then  since $\s(v)\geq 2k-1$, each $v\in I$
 is adjacent to
each $w\in V(G)-I$. Thus
\begin{equation}\label{l4}
\mbox{every two big sets in any extremal $G$ are disjoint.}
\end{equation}

\subsection{Preliminaries and main result}
Since every cycle in a simple graph has at least $3$ vertices, the condition $|G|\geq3k$ is necessary in Theorem~\ref{cht}.
However, it is not necessary for multigraphs, since loops and multiple edges form cycles with fewer than three vertices. 
Theorem~\ref{cht} can easily be extended to multigraphs, although  the statement is no longer as simple:


 \begin{theorem}\label{chm}
 For $k \in \mathbb Z^+$,
let $G$ be a multigraph with $\ms(G) \geq 2k$, and set $F=F(G)$ and $\alpha'=\alpha'(F)$.
Then $G$ has no $k$ disjoint cycles if and only if
\begin{equation}\label{eqn:chm}|V(G)|-|V_1(G)|-2\alpha'<3(k-|V_1|-\alpha'),\end{equation}
i.e.,
 $ |V(G)|+2|V_1|+\alpha'<3k$.  
  \end{theorem}
\begin{proof}
If (\ref{eqn:chm}) holds, then $G$ does not have enough vertices to contain $k$ disjoint cycles.
If (\ref{eqn:chm}) fails, then we choose $|V_1|$ cycles of length one and $\alpha'$ cycles of length  two from $V_1 \cup V(F)$.
 By Theorem~\ref{cht}, the remaining (simple) graph contains $k-|V_1|-\alpha'$ disjoint cycles.
\end{proof}
Theorem~\ref{chm} yields the following.

 \begin{cor}\label{dm1}
Let $G$ be a multigraph with $\ms(G)\geq 2k-1$ for some integer $k \geq 2$, and set $F=F(G)$ and $\alpha'=\alpha'(F)$.
 Suppose $G$ contains at least one loop.
Then $G$ has no $k$ disjoint cycles if and only if
 $ |V(G)|+2|V_1|+\alpha'<3k$.
  \end{cor}

Instead of the $(2k-1)$-connected multigraphs of Question~\ref{DiracQ},
we consider the wider family
 $\mathcal{D}_k$.
Since acyclic graphs are exactly forests, Theorem~\ref{ch++} can be restated as follows:

 \begin{theorem}\label{kky}
 For $k \in \mathbb Z^+$,
let $G$ be a simple graph in $\mathcal{D}_k$.
Then $G$ has no $k$ disjoint cycles if and only if  one of the following holds:
\begin{enumerate}
\item[$(\alpha)$] $|G|\leq 3k-1$;
\item[$(\beta)$] $k=1$ and $G$ is a forest with no isolated vertices;
\item[$(\gamma)$] $k=2$ and $G$ is a wheel;
\item[$(\delta)$] $\alpha(G)= n-2k+1$; or
\item[$(\epsilon)$] $k>1$ is odd and $G={Y_{k,k}}$.
\end{enumerate}
  \end{theorem}



 Dirac~\cite{Di}  described all multigraphs in $\mathcal D_2$ that do not have two disjoint cycles:

 \begin{theorem}[\cite{Di}]\label{dt} 
 Let $G$ be a $3$-connected multigraph. Then $G$ has no two disjoint cycles if and only if one of the following holds:
 \begin{enumerate}
\item[(A)]  $\widetilde{G}=K_4$ and the strong edges in $G$ form either a star (possibly empty) or
a $3$-cycle;
\item[(B)] $G=K_5$;
\item[(C)]  $\widetilde{G}=K_5-e$ and the strong edges in $G$ are not incident to the ends of $e$;
\item[(D)] $\widetilde{G}$ is a wheel, where some spokes could be strong edges; or
\item[(E)] $G$ is obtained from $K_{3,|G|-3}$ by adding non-loop edges between the vertices of the (first) 3-class.
\end{enumerate}
  \end{theorem}

Going further, Lov\'asz~\cite{Lo} described \emph{all} multigraphs with no two disjoint cycles. He observed that 
it suffices
to describe such multigraphs with minimum (ordinary) degree at least $3$, and proved the following:

\begin{theorem}[\cite{Lo}]\label{lovasz}
Let $G$ be a multigraph  with $\delta(G) \geq 3$. Then $G$ has no two disjoint cycles if and only if  $G$ is one of the following: 
\begin{enumerate}
\item[$(1)$] $K_5$;
\item[$(2)$] A wheel, where some spokes could be strong edges;
\item[$(3)$] $K_{3,|G|-3}$ together with a loopless multigraph on the vertices of the (first) 3-class; or
\item[$(4)$]
a forest $F$ and a vertex $x$ with possibly some loops at $x$ and some
 edges linking $x$ to $F$.
 \end{enumerate}
\end{theorem}

By Corollary~\ref{dm1}, in order to describe the multigraphs in $\mathcal{D}_k$ not containing $k$ disjoint cycles,
it is enough to describe such multigraphs with no loops.
Our main result is the following:

 \begin{theorem}\label{dm}
 Let $k\geq 2$ and $n\geq k$ be integers.
 Let $G$ be an $n$-vertex multigraph in $\mathcal D_k$ with no loops.
Set $F=F(G)$, $\alpha'=\alpha'(F)$, and $k'=k-\alpha'$.
Then $G$ does not contain $k$ disjoint cycles if and only if one of the following holds: (see Figure~\ref{fig:dm})
\begin{enumerate}[label=(\alph*)]
\item  $ n+\alpha'<3k$;
\item $|F|=2\alpha'$ (i.e., $F$ has a perfect matching) and either\\
 (i)
$k'$ is odd and $G-F=Y_{k',k'}$, or\\
(ii)  $k'=2<k$ and $G-F$ is a wheel with  $5$ spokes;
\item $G$ is extremal  and either\\
(i) some big set is not incident to any strong edge, or\\
(ii) for
some two distinct big sets $I_j$ and $I_{j'}$, all strong edges intersecting $I_j\cup I_{j'}$ have a common vertex
outside of $I_j\cup I_{j'}$;
\item $n=2\alpha'+3k'$, $k'$ is odd, and $F$ has a superstar $S=\{v_0,\ldots,v_s\}$ with center $v_0$ such that
either\\
(i)
$G-(F-S+v_0)=Y_{k'+1,k'}$, or\\
(ii) $s=2$, $v_1v_2\in E(G)$,
$G-F=Y_{k'-1,k'}$ and $G$ has no edges between $\{v_1,v_2\}$ and the set $X_0$ in $G-F$;
\item $k=2$ and $G$ is a wheel, where some spokes could be strong edges;
\item $k'=2$, $|F|=2\alpha'+1=n-5$, and $G-F=C_5$.
\end{enumerate}
 \end{theorem}
 
\begin{figure}[ht]\centering
\captionsetup[subfigure]{labelformat=empty}
\begin{subfigure}[b]{.3\textwidth}
\centering
\begin{tikzpicture}
\draw node[shape=circle,fill=black,inner sep=0pt,minimum size=2mm,draw] (c){};
\foreach \x in {0,1,2}{
\draw (90+120*\x:1cm) node[shape=circle,fill=black,inner sep=0pt,minimum size=2mm,draw] (t\x){};
\draw (t\x)--(c);
}
\draw[thick, double distance=2pt] (t1)--(t2)--(t0)--(t1);
\end{tikzpicture}\caption{(a): $\alpha'=1$, $k=2$}
\end{subfigure}
\begin{subfigure}[b]{.3\textwidth}
\centering
\begin{tikzpicture}
\foreach \x in {0,1,2}{\foreach \y in {0,1,2}
{
\draw (\x,.75*\y) node[shape=circle,fill=black,inner sep=0pt,minimum size=2mm,draw] (d\x\y){};
}}
\draw (d12)+(0,.25) node (d0){};
\foreach \y in {0,2}{\foreach \x in {0,1,2}{\foreach \z in {0,1,2}{
\draw (d\x1)--(d\z\y);
}}
\draw (d0\y)--(d1\y)--(d2\y);}
\draw (d00) to[out=-30, in=210] (d20);
\draw (d02) to[out=30, in=-210] (d22);
\draw (2.75,.25) node[shape=circle,fill=black,inner sep=0pt,minimum size=2mm,draw] (d1){};
\draw (2.75,1.25) node[shape=circle,fill=black,inner sep=0pt,minimum size=2mm,draw] (d2){};
\draw[thick, double distance=2pt] (d1)--(d2);
\draw (3.5,.25) node[shape=circle,fill=black,inner sep=0pt,minimum size=2mm,draw] (d3){};
\draw (3.5,1.25) node[shape=circle,fill=black,inner sep=0pt,minimum size=2mm,draw] (d4){};
\draw[thick, double distance=2pt] (d3)--(d4);
\end{tikzpicture}\caption{(b)(i): $\alpha'=2$, $k=5$}
\end{subfigure}
\begin{subfigure}[b]{.3\textwidth}
\centering
\begin{tikzpicture}
\draw (0,0) node[shape=circle,fill=black,inner sep=0pt,minimum size=2mm,draw] (c){};
\foreach \x in {0,...,4}{
\draw (18+72*\x:1cm) node[shape=circle,fill=black,inner sep=0pt,minimum size=2mm,draw] (d\x){}; 
\draw (d\x)--(c);}
\foreach \x in {0,...,3}{
\setcounter{next}{\x}
\addtocounter{next}{1}
\draw (d\arabic{next}) --(d\x);}
\draw (d4)--(d0);
\draw (1.74,-.5) node[shape=circle,fill=black,inner sep=0pt,minimum size=2mm,draw] (d1){};
\draw (1.74,.5) node[shape=circle,fill=black,inner sep=0pt,minimum size=2mm,draw] (d2){};
\draw[thick, double distance=2pt] (d1)--(d2);
\draw (2.5,-.5) node[shape=circle,fill=black,inner sep=0pt,minimum size=2mm,draw] (d3){};
\draw (2.5,.5) node[shape=circle,fill=black,inner sep=0pt,minimum size=2mm,draw] (d4){};
\draw[thick, double distance=2pt] (d3)--(d4);
\end{tikzpicture}\caption{(b)(ii): $\alpha'=2$,  $k=4$}
\end{subfigure}
\vskip 7mm
\begin{subfigure}[b]{.3\textwidth}
\centering
\begin{tikzpicture}
\draw (0,0) node[shape=ellipse, minimum width=1cm, minimum height=2.5cm, draw] (le){};
\draw (2,0) node[shape=ellipse, minimum width=1cm, minimum height=2.5cm, draw] (re){};
\foreach \y in {0,...,3}{
\foreach \x in {0,2}{
\draw (\x,.5*\y-.75) node[shape=circle,fill=black,inner sep=0pt,minimum size=2mm,draw] (n\x\y){};}}
\draw (1,1.75)node[shape=circle,fill=black,inner sep=0pt,minimum size=2mm,draw](v){};
\draw[thick, double distance=2pt] (n22)--(v);
\draw[thick, double distance=2pt] (n23)--(v);
\draw[thick, double distance=2pt] (n03)--(v);
\draw (n01)+(-1,0) node{$I_j$};
\draw (n21)+(1,0) node{$I_{j'}$};
\end{tikzpicture}\caption{(c)(ii)}
\end{subfigure}
\begin{subfigure}[b]{.3\textwidth}
\centering
\begin{tikzpicture}
\foreach \x in {-2,...,2}{
\draw (\x,0) node[shape=circle,fill=black,inner sep=0pt,minimum size=2mm,draw] (m\x) {};
}
\foreach \x in {0,1,2}{
\draw (\x-2,1.25) node[shape=circle,fill=black,inner sep=0pt,minimum size=2mm,draw] (u\x) {};
\draw (\x-2,-1.25) node[shape=circle,fill=black,inner sep=0pt,minimum size=2mm,draw] (l\x) {};
}
\draw (1,1) node[shape=circle,fill=black,inner sep=0pt,minimum size=2mm,draw, label=above:$v_0$] (v0){};
\draw (2,1) node[shape=circle,fill=black,inner sep=0pt,minimum size=2mm,draw] (v3){};
\draw[thick, double distance=2pt] (m2)--(v3);
\draw[thick, double distance=2pt] (m1)--(v0);
\draw[thick, double distance=2pt] (m0)--(v0);
\draw (m0) node[inner sep=0.7mm, label= right:$v_1$] (v1){};
\draw (m1) node[inner sep=0.7mm,label= right:$v_2$] (v2){};
\foreach \m in {-2,...,1}{\foreach \u in {0,1,2}{
\draw[gray] (l\u)--(m\m)--(u\u);}}
\draw[gray] (u0)--(u1)--(u2);
\draw[gray] (l0)--(l1)--(l2);
\draw[gray] (l0) to[out=-45, in=225] (l2);
\draw[gray] (u0) to[out=45, in=-225] (u2);
\end{tikzpicture}\caption{(d)(i): $\alpha'=2$,  $k=5$}
\end{subfigure}
\begin{subfigure}[b]{.3\textwidth}
\centering
\begin{tikzpicture}
\foreach \x in {-2,...,2}{
\draw (\x,0) node[shape=circle,fill=black,inner sep=0pt,minimum size=2mm,draw] (m\x) {};
}
\foreach \x in {0,1,2}{
\draw (\x-2.5,1.25) node[shape=circle,fill=black,inner sep=0pt,minimum size=2mm,draw] (u\x) {};
\draw (\x-2.5,-1.25) node[shape=circle,fill=black,inner sep=0pt,minimum size=2mm,draw] (l\x) {};
}
\draw (0.5,1) node[shape=circle,fill=black,inner sep=0pt,minimum size=2mm,draw, label=above:$v_0$] (v0){};
\draw (2,1) node[shape=circle,fill=black,inner sep=0pt,minimum size=2mm,draw] (v3){};
\draw[thick, double distance=2pt] (m0)--(v0);
\draw[thick, double distance=2pt] (m1)--(v0);
\draw[thick, double distance=2pt] (m2)--(v3);
\draw (m0) node[inner sep=1mm, label=below:$v_1$] (v1){};
\draw (m1) node[inner sep=1mm, label=below:$v_2$] (v2){};
\draw (v1)--(v2);
\foreach \m in {-2,-1}{\foreach \u in {0,1,2}{
\draw[gray] (l\u)--(m\m)--(u\u);}}
\draw[gray] (u0)--(u1)--(u2);
\draw[gray] (l0)--(l1)--(l2);
\draw[gray] (l0) to[out=-45, in=225] (l2);
\draw[gray] (u0) to[out=45, in=-225] (u2);
\end{tikzpicture}\caption{(d)(ii): $\alpha'=2$,  $k=5$}
\end{subfigure}
\vskip 7mm
\begin{subfigure}[b]{.3\textwidth}
\centering
\begin{tikzpicture}
\draw (0,0) node[shape=circle,fill=black,inner sep=0pt,minimum size=2mm,draw] (c) {};
\foreach \x in {0,...,5}
{\draw (c)+(60*\x:1.25cm) node[shape=circle,fill=black,inner sep=0pt,minimum size=2mm,draw] (d\x){};
\draw[thick, double distance=2pt] (c)--(d\x);}
\foreach \x in {0,...,4}{
\setcounter{next}{\x}
\addtocounter{next}{1}
\draw (d\arabic{next}) --(d\x);}
\draw (d5)--(d0);
\end{tikzpicture}\caption{(e): $\alpha'=1$,  $k=2$}
\end{subfigure}
\begin{subfigure}[b]{.3\textwidth}
\centering
\begin{tikzpicture}
\foreach \x in {0,...,4}
{\draw (90+72*\x:.75cm) node[shape=circle,fill=black,inner sep=0pt,minimum size=2mm,draw] (d\x){};
}
\foreach \x in {0,...,3}{
\setcounter{next}{\x}
\addtocounter{next}{1}
\draw (d\arabic{next}) --(d\x);}
\draw (d4)--(d0);
\draw (2,-1) node(f) {};
\draw (f)+(-2,0) node(f1) {};
\draw (f1)+(1.5,.5) node[shape=circle,fill=black,inner sep=0pt,minimum size=2mm,draw] (e1){};
\draw (f1)+(1.5,1.5) node[shape=circle,fill=black,inner sep=0pt,minimum size=2mm,draw] (e2){};;
\draw[thick, double distance=2pt] (e1)--(e2);
\draw (f1)+(2.5,.5) node[shape=circle,fill=black,inner sep=0pt,minimum size=2mm,draw] (e3){};
\draw (f1)+(2.5,1.5) node[shape=circle,fill=black,inner sep=0pt,minimum size=2mm,draw] (e4){};
\draw[thick, double distance=2pt] (e3)--(e4);
\draw (f1)+(3.25,1.25) node[shape=circle,fill=black,inner sep=0pt,minimum size=2mm,draw] (e5){};
\draw[thick, double distance=2pt] (e3)--(e5);
\end{tikzpicture}\caption{(f): $\alpha'=2$, $k=4$}
\end{subfigure}
\caption{Examples of Subgraphs of Multigraphs Listed in Theorem~\ref{dm}}\label{fig:dm}
\end{figure}
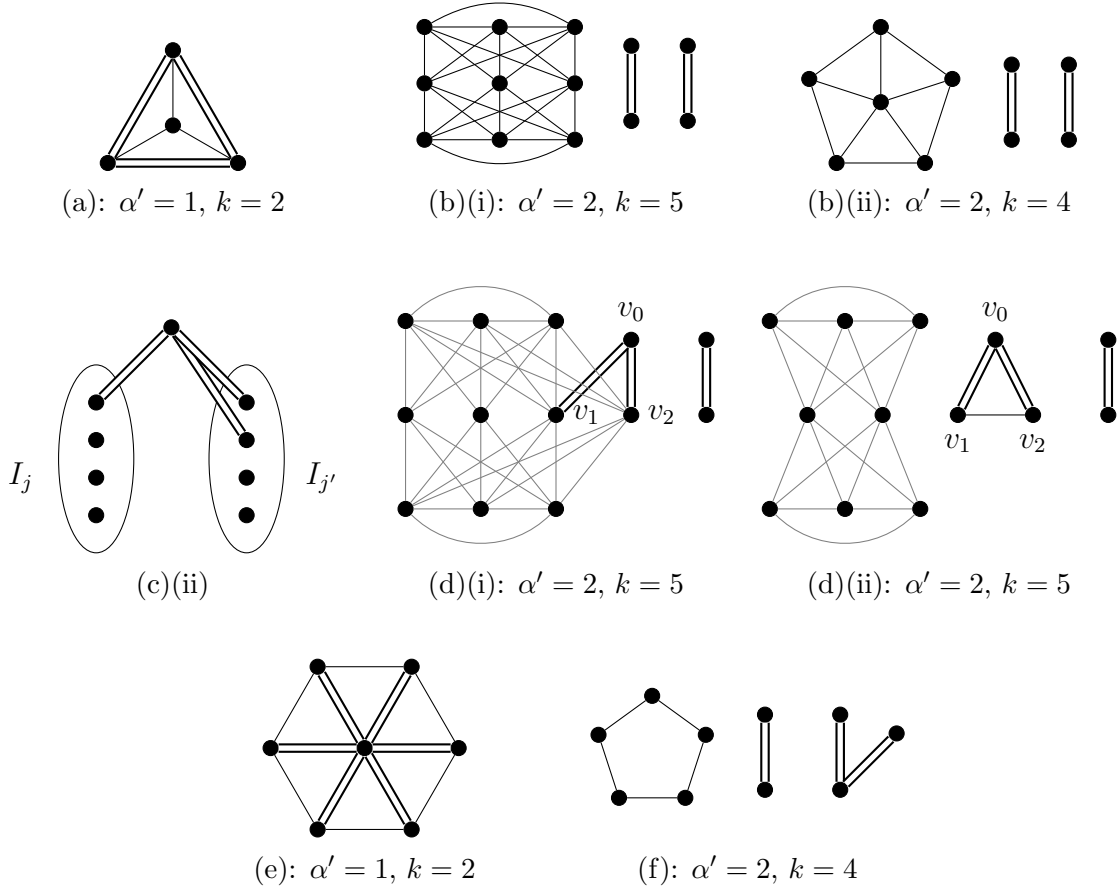
The six infinite classes of multigraphs described in Theorem~\ref{dm} are exactly the family of multigraphs in $\mathcal D_k$ with no $k$ disjoint cycles. So, the $(2k-1)$-connected multigraphs with no $k$ disjoint cycles are exactly the $(2k-1)$-connected multigraphs that are in one of these classes. For any multigraph $G$, we can  check in polynomial time whether $G\in\mathcal{D}_k$ and whether $G$ is $(2k-1)$-connected.
If $G \in \mathcal D_k$,
we can
check in polynomial time whether any of the conditions (a)--(f)  hold for $G$. Note that
to determine the extremality of $G$ we 
need only check whether $G$ has 
an independent set of size $n-2k+1$. Such a set will be the complement of $N(v)$ for some vertex $v$ with 
$\s(v)=2k-1$; so all big sets can be found in polynomial time.

Note if $G$ is $(2k-1)$-connected, and (b)(i), (d)(i), or d(ii) holds, then $k' \le 1$.


\section{Proof of sufficiency in  Theorem~\ref{dm} }
 Suppose $G$ 
  has a set $\mathcal C$ of $k$ disjoint cycles. Our task is to show that each of (a)--(f) fails.  Theorem~\ref{lovasz}, case (2) implies (e) fails. Let $M\subseteq\mathcal C$ be the set of strong edges  ($2$-cycles) in $\mathcal C$, $h=| M|$, and $W=V(M)$. Now $h\le\alpha'$; so $n\ge 2h+3(k-h)\ge3k-\alpha'$. Thus (a) fails. If  $n=3k-\alpha'$ as in cases (b), (d) and (f), then $h=\alpha'$ and $G'=G-W$ is a simple graph of minimum degree at least $2k'-1$ with $3k'$ vertices and $k'$ cycles. By Theorem~\ref{ch++} all of (i)--(iii) hold for $G'$.
  In case (b),  $G'=G-F$; so (ii) and (iii) imply (b)(i) and (b)(ii)  fail. In case (f), $G'=G-(F-v)= v\vee C_5$ for some vertex  $v\in F$. So (iii) implies (f) fails. In case (d), $M$ consists of a strong perfect matching in $F-S$ together with a strong edge $v_0v\in S$. 
   If $G-(F-S+v_0)=Y_{k'+1,k'}$ then either $\alpha(G')=k'+1$ or $G'=Y_{k',k'}$,  contradicting (i) or (ii). So (d)(i) fails. Similarly, in case (d)(ii), $G'\subseteq Y_{k',k'}$, another contradiction.
  
In case (c), $G$ is extremal. Every big set $I$ 
satisfies $|V(G)-I|<2k$.  So some cycle  $C_I\in\mathcal C$ has at most one vertex in $V(G)-I$. Since $I$ is independent, $C_I$ has at most one vertex in $I$. Thus $C_I$
is a strong edge and $(c)(i)$ fails. 
Let $J$ be another big set; then $I\cap J=\emptyset$. As cycles in $\mathcal C$ are disjoint,  $C_I=C_J$ or $C_I\cap C_J=\emptyset$. 
Regardless, $C_I\cap C_J\subseteq I\cup J$. So (c)(ii) fails.

\section{Proof of necessity in  Theorem~\ref{dm} }
 Suppose   $G$ does not have $k$ disjoint cycles. Our goal is to show that one of (a)--(f) holds. If $k=2$ then one of the cases (1)--(4) of Theorem~\ref{lovasz}
 holds. If (1) holds then $\alpha'=0$, and so  (a) holds. Case  (2) is (e). Case (3) yields (c)(i), where  the partite set of size $n-3$ is the big set. As $G\in \mathcal D_k$,  it has no vertex  $l$ with $s(l)<3$. So (4) fails, because each leaf $l$ of the forest satisfies $s(l)\le2$.
Thus below we assume
\begin{equation}\label{l41}
k\geq 3.
\end{equation}

Choose a maximum strong matching $M\subseteq F$ with $\alpha(G-W)$  minimum, where $W=V(M)$.
Then $|M|=\alpha'$, $G':=G-W$ is simple, and $\delta(G')\ge2k-1-2\alpha'= 2k'-1$. So $G'\in \mathcal D_{k'}$. Let $n':=|V(G')|=n-2\alpha'$.  
Since $G'$ has no $k'$ disjoint cycles, Theorem~\ref{kky} implies one of the following:
$(\alpha)$~$|G'|\leq 3k'-1$;
$(\beta)$~$k'=1$ and $G'$ is a forest with no isolated vertices;
$(\gamma)$~$k'=2$ and $G'$ is a wheel;
$(\delta)$~$\alpha(G')= n'-2k'+1=n-2k+1$; or
$(\epsilon)$~$k'>1$ is odd and $G'=Y_{k',k'}$.
If $(\alpha)$ holds then so does  (a). So suppose $n'\ge3k'$. In the following we may obtain a contradiction by showing $G$ has $k$ disjoint cycles.


{\bf Case 1:} $(\beta)$ holds. 
By~\eqref{l41}, there are strong edges $yz,y'z'\in M$. As $\mathcal S(G)\ge2k-1$, each vertex $v\in V(G')$ is adjacent
to all but $d_{G'}(v)-1$ vertices of $W$. 

{\em Case 1.1:} $G'$ contains a path on four vertices, or $G'$ contains at least two  components. Let $P=x_1\dots x_t$ be a maximum path in $G'$. Then $x_1$ is a leaf in $G'$, and either $d_{G'}(x_2)=2$ or $x_2$ is adjacent to a leaf $l\ne x_1$. So $vx_1x_2v$ or  $vx_1x_2lv$ is a cycle for all but at most one vertex $v\in W$. If $t \geq 4$, let $s_1=x_t$ and $s_2=x_{t-1}$. Otherwise, $G'$ is disconnected and every component is a star; in a component not containing $P$, let $s_1$ be a leaf and let $s_2$ be its neighbor. As before, for all but at most one vertex $v'\in W$, either $v's_1s_2v'$ is a cycle or  $v's_1s_2l'v'$ is a cycle for some leaf $l'$. Thus $G[(V\smallsetminus W)\cup\{u,v\}]$ contains two disjoint cycles for some $uv\in\{yz,y'z'\}$. These cycles and the $\alpha'-1$ strong edges of $M-uv$ yield $k$ disjoint cycles in $G$, a contradiction. 



{\em Case 1.2:} $G'$ is  a star with center $x_0$ and leaf set $X=\{x_1,x_2,\ldots, x_t\}$.
Since $n'\ge3k'$, $t\ge2$ and $X$ is a big set in $G$. If (c)(i) fails then some vertex in $X$, say $x_1$, is incident to a strong edge, say
 $x_1y$. If $t\geq 3$, then $G$ has $k$ disjoint cycles: $|M-yz+yx_1|$ strong edges and $zx_2x_0x_3z$. 
Else $t=2$. Then  $n=3\alpha'+3k'=2k+1$, as in (d); and each vertex of $G$ is adjacent to all but at most one other vertices. If $x_0z\in E(G)$ then again $G$ has $k$ disjoint cycles:
  $|M-yz+yx_1|$ strong edges and  $zx_0x_2z$, a contradiction. So $N(x_0)=V(G)-z-x_0$, and $G[\{x_0,x_1,x_2,z\}]=C_4=Y_{2,1}$. Also $y$ is the only possible strong neighbor of $x_1$ or $x_2$: if $u\in\{x_1,x_2\}$, $y'z'\in M$ with $y'\ne y$ (maybe $y'=z$) and $uy'\in E(F)$, using the same argument as above, if $z'x_0 \in E(G)$ then $G$ has $k$ disjoint cycles consisting of  $|M-y'z'+y'u|$ strong edges and $G[G'-u+z']$, a contradiction.
Then   $x_0z' \not\in E(G)$, so $z'=z$, and $y'=y$.
Thus $S=N_F(y)\cap\{z,x_0,x_1,x_2\}+y$ is a superstar. So (d)(i) holds.


{\bf Case 2:}  $(\gamma)$ holds. Then $k'=2$ and $G'$ is a wheel with center $x_0$ and rim $x_1x_2\ldots x_tx_1$.
By~\eqref{l41}, there exists $yz\in M$. Since (a) fails, $t\geq 5$.
For $i \in [t]$, 
$$s(x_i) \geq 2k-1 = 2\alpha'+3=2\alpha'+|N(x_i) \cap G'|,$$
 so $x_i$ is adjacent to every vertex in $W$. If $t \geq 6$, then $G'$ has $k$ disjoint cycles:  $|M-yz|$ strong edges, $yx_1x_2y$, $zx_3x_4z$ and $x_0x_5x_6x_0$. Thus $t=5$. If no vertex of $G'$ is incident to a strong edge, then (b)(ii) holds. Therefore, we assume $y$ has a strong edge to $G'$. The other endpoint of the strong edge could be in the outer cycle, or could be $x_0$. If some vertex in the outer cycle, say $x_1$, has a strong edge to $y$, then we have $k$ disjoint cycles: $|M-yz+yx_1|$ strong edges, $zx_2x_3z$ and $x_0x_4x_5x_0$. The
last possibility is that $x_0$ has a strong edge to $y$, and (f) holds.

{\bf Case 3:} $(\epsilon)$ holds.
Then $k'>1$ is odd,  $G'=Y_{k',k'}$ and $n=2\alpha'+3k'$.
Let $X_0=\{x_1,\ldots,x_{k'}\}$, $X_1=\{x'_1,\ldots,x'_{k'}\}$,
and $X_2=\{x''_1,\ldots,x''_{k'}\}$ be the sets from the definition of $Y_{k',k'}$. 
Observe
\begin{equation}\label{*} 
 \overline K_{s+t}\vee(K_{2s}\cup K_{2t})
 \textrm{ contains  $s+t$ disjoint triangles.}
\end{equation}
By degree conditions, each $x'\in X_1\cup X_2$ is adjacent to each $v\in W$ and each $x\in X_0$ is adjacent
to all but at most one $y\in W$. 
If (b)(i) fails then some strong edge
$uy$  is incident with a vertex $u\in V(G')$. If possible, pick $u\in X_1\cup X_2$. By symmetry we may assume $u\notin X_2$. Let $yz$ be the edge of $M$ incident to $y$.
Set $v_0=y$ and $\{v_1,\ldots,v_s\}
=V(F \cap G')+z$.
We will prove that $\{v_0,\ldots,v_s\}$ is a superstar, and use this to show that (d)(i) or (d)(ii) holds. Let $G^*=G-(W-z)$, and observe that $Y_{k'+1,k'}$ is a spanning subgraph of $G^*$ with equality if $X_0+z$ is independent. 

Suppose $xz\in E(G)$ for some $x\in X_0-u$. Then $G$ has $k$ disjoint cycles: $|M-yz+yu|$ strong edges, $zxx''_1z$, and $k'-1$ disjoint cycles in $G^*-\{x,x''_1,u\}$, obtained by applying \eqref{*} directly if $u\in X_1$, or by using $T:=x'_1x'_2x'_3x'_1$ and applying \eqref{*} to $G^*-\{x,x''_1,u\}-T$ if $u\in X_0$. This contradiction implies  $zu$ is the only possible edge in $G[X_0+z]$. Thus if $y$ has two strong neighbors in $X_0$ then $X_0+z$ is independent, and $G^*=K_{k'+1,k'}$. Also by degree conditions, every $x\in X_0-u$ is adjacent to every $w\in W-z$. 
So if $y'z'\in M$ with $y'\ne y$ and $u'\in V(G')$, then $u'y'\notin E(F)$: else $x\in X_0-u-u'$ satisfies $xz'\in E(G)$ and $xz'\notin E(G)$. 
 So $\{v_0,\ldots, v_s\}$ is a superstar. If $X_0+z$ is independent then (d)(i) holds; else (d)(ii) holds. 

{\bf Case 4:} 
 $(\delta)$ holds. Then $\alpha(G')= n'-2k'+1>n'/3$, since $n' \geq 3k'$. 
  So $G'$ is extremal. Let $J$ be a big set in $G'$. Then $|J|=n'-2k'+1= n-2k+1$. So  $G$ is extremal and $J$ is a big set in $G$. 
Also each $x\in J$ is adjacent to every $y\in V(G)-J$.
If (c)(i) fails then some $x\in J$ has a strong neighbor $y$.
Let $yz$ be the edge in $M$ containing $y$. In $F$, consider the maximum matching $M'=M-yz+xy$, and set $G''=G-V(M')$. By the choice of $M$,
$G''$ contains a big set $J'$, and $J'$ is big in $G$. Since $x\notin J'$, \eqref{l4} implies $J'\cap J=\emptyset$ (possibly, $z\in J'$).
If (c)(ii) fails then there is a strong edge $vw$ such that $v\in J\cup J'$ and $w\neq y$. Moreover, by the symmetry
between  $J$ and $J'$,
we may assume  $v\in J'$. Let $uw$ be the edge in $M$ containing $w$.
 Since $M$ is maximum, $u\neq z$. Let $M''=M'-uw+vw$.
Again by the case, $G-V(M'')$ contains a big set  $J''$. Since $x,v \not\in J''$, $J''$ is disjoint from $J\cup J'$. So $n'\ge3|J|>n'$,
 a contradiction.
\qed

\noindent {\bf Acknowledgment.} We thank Mikl\' os Simonovits for attracting our attention to Dirac's paper and for many helpful discussions. We also thank a referee for thoughtful suggestions that improved our article.


\begin{thebibliography}{10}

\bibitem{CH} K. Corr\' adi and A. Hajnal, On the maximal number of independent
circuits in a graph. {\em Acta Math. Acad. Sci. Hungar.} 14 (1963)
423--439.

\bibitem{Di} G. Dirac, Some results concerning the structure of graphs, Canad. Math. Bull.  6 (1963) 183--210.

\bibitem{DE}G. Dirac and P. Erd\H os, On the maximal number of independent circuits in a graph,
{\em Acta Math. Acad. Sci. Hungar.} 14 (1963) 79--94.

\bibitem{KK2}H. A. Kierstead, A. V. Kostochka, and E. C. Yeager, On the Corr\' adi-Hajnal Theorem and a question of Dirac,
submitted.

\bibitem{Lo} L. Lov\' asz, On graphs not containing independent circuits, (Hungarian. English summary)
Mat. Lapok 16 (1965), 289--299.

\end{thebibliography}
\end{document}